\documentclass[a4paper,12pt,reqno]{amsart}
\usepackage{latexsym}
\usepackage{amssymb} 
\usepackage{mathrsfs}
\usepackage{amsmath}
\usepackage{latexsym}
\usepackage{delarray}
\usepackage{amssymb,amsmath,amsfonts,amsthm,mathrsfs}
\usepackage{blkarray}

\usepackage{hyperref}
\renewcommand\eqref[1]{(\ref{#1})} 
 \newtheorem{thm}{Theorem}[section]
 \newtheorem{cor}[thm]{Corollary}
 
 \newtheorem{prop}[thm]{Proposition}
 \theoremstyle{definition}
 
 \theoremstyle{remark}
 \newtheorem{rem}[thm]{Remark}
 
 \numberwithin{equation}{section}
\newcommand{\half}{\frac{1}{2}}

\newcommand{\ene}{\mathbb{N}}

\newcommand{\er}{\mathbb{R}}
\newcommand{\ce}{\mathbb{C}}

\newcommand{\efee}{\mathcal{F}}

\newcommand{\bi}{\begin{itemize}}
\newcommand{\cinfm}{{C}^{\infty}(M)}

\newcommand{\ei}{\end{itemize}}
\newcommand{\be}{\begin{enumerate}}
\newcommand{\ee}{\end{enumerate}}
\newcommand{\beq}{\begin{equation}}
\newcommand{\eq}{\end{equation}}

\newcommand{\Dcal}{\mathcal{D}}

\def\p#1{{\left({#1}\right)}}

\def\En{{E_{o}}}

\def\Hcal{{\mathcal H}}

\DeclareMathOperator{\Tr}{Tr}

\def\HS{{\mathtt{HS}}}
\def\Rn{{{\mathbb R}^n}}
\def\Tn{{{\mathbb T}^n}}
\def\Zn{{{\mathbb Z}^n}}

\def\N{{{\mathbb N}}}
\def\C{{{\mathbb C}}}
\def\SU2{{{\rm SU(2)}}}
\def\lapsu2{{{\mathcal L}_\SU2}}

\begin{document}

%
%
%
%
%
%
%
%
%
\title[Fourier multipliers in Hilbert spaces]
 {Fourier multipliers in Hilbert spaces}

\author[Julio Delgado]{Julio Delgado}

\address{%
Department of Mathematics\\
Imperial College London\\
180 Queen's Gate, London SW7 2AZ\\
United Kingdom
}
\email{j.delgado@imperial.ac.uk}

\thanks{The authors were supported by the Leverhulme Grant RPG-2017-151.}
\author[Michael Ruzhansky]{Michael Ruzhansky}

\address{%
Department of Mathematics\\
Imperial College London\\
180 Queen's Gate, London SW7 2AZ\\
United Kingdom
}

\email{m.ruzhansky@imperial.ac.uk}



\date{\today}
\begin{abstract}
This is a survey on a notion of invariant operators, or Fourier multipliers on Hilbert spaces. This concept is defined with respect to a fixed partition of the space into a direct sum of finite dimensional subspaces. In particular this notion can be applied to the important case of $L^2(M)$ where $M$ is a compact manifold $M$ endowed with a positive measure.  The partition in this case comes from  
 the spectral properties of a a fixed 
elliptic operator $E$. 
\end{abstract}

\maketitle
\tableofcontents

\section{Introduction}

{\em These notes are based on our paper \cite{dr16:fmm} and have been prepared for the instructional volume associated to the Summer School on Fourier Integral Operators held in Ouagadougou, Burkina Faso, where the authors took part during 14--26 September 2015.}

In this note we discuss invariant operators, or Fourier 
multipliers in a general Hilbert space $\Hcal$. This notion is based on a 
partition of $\Hcal$ into a direct sum of finite dimensional subspaces, so that
a densely defined operator on $\Hcal$ can be decomposed as acting in these
subspaces. In the present exposition we follow our detailed description in \cite{dr16:fmm}, with which there are intersections and to which we also refer for further details. 

There are two main examples of this construction: operators on $\Hcal=L^{2}(M)$ for a compact manifold $M$ as well as 
operators on $\Hcal=L^{2}(G)$ for a compact Lie group $G$. The difference in 
approaches to these settings is in the choice of partitions of $\Hcal$ into direct sums of
subspaces: in the former case they are chosen as eigenspaces of a fixed
elliptic pseudo-differential operator 
on $M$ while in the latter case they are chosen as linear spans 
of matrix coefficients of inequivalent irreducible unitary representations of $G$.

Let $M$ be a closed manifold (i.e. a  compact smooth manifold without boundary) of 
dimension $n$ endowed with a positive measure $dx$. 
Given an elliptic positive pseudo-differential operator $E$ of order $\nu$ on $M$, 
by considering an orthonormal basis consisting of eigenfunctions of $E$ we can  
associate a discrete Fourier analysis to the operator $E$ in the sense introduced 
by Seeley (\cite{see:ex}, \cite{see:exp}). 
 
These notions can be applied to the derivation of conditions 
characterising those invariant operators on
$L^{2}(M)$ that belong to Schatten classes.
Furthermore, sufficient conditions for the $r$-nuclearity on $L^{p}$-spaces can also be obtained as well as the corresponding trace formulas relating operator traces to expressions involving their symbols. More details on these applications can be found in Section 8 of \cite{dr16:fmm}.

A characteristic feature that appears is that no regularity is assumed
neither on the symbol nor on the kernel. 
In the case of compact Lie groups, our results extend results on
Schatten classes and on $r$-nuclear operators on $L^p$ spaces that have been
obtained in  \cite{dr13:schatten} and \cite{dr13a:nuclp}. This can be shown by relating the symbols introduced in this paper to matrix-valued symbols on compact
Lie groups developed in 
\cite{rt:groups} and in \cite{rt:book}.

To formulate the notions more precisely, let $\Hcal$ be a complex Hilbert space and let 
$T:\Hcal\rightarrow \Hcal$ be a linear compact operator. If we denote by 
$T^*:\Hcal\rightarrow \Hcal$ the adjoint of  $T$, then the linear operator 
$(T^*T)^\half:\Hcal\rightarrow \Hcal$ is positive and compact. 
Let $(\psi_k)_k$ be an orthonormal basis for $\Hcal$ consisting of eigenvectors of 
$|T|=(T^*T)^\half$, and let $s_k(T)$ be the eigenvalue corresponding to the eigenvector 
$\psi_k$, $k=1,2,\dots$. The non-negative numbers $s_k(T)$, $k=1,2,\dots$, 
are called the singular values of $T:\Hcal\rightarrow \Hcal$. 
If $0<p<\infty$ and the sequence of singular values is $p$-summable, then $T$ 
is said to belong to the Schatten class  ${S}_p(\Hcal)$, and it is well known that each 
${S}_p(\Hcal)$ is an ideal in $\mathscr{L}(\Hcal)$. 
If $1\leq p <\infty$, a norm is associated to ${S}_p(\Hcal)$ by
 \[
 \|T\|_{S_p}=\left(\sum\limits_{k=1}^{\infty}(s_k(T))^p\right)^{\frac{1}{p}}.
 \] 
If $1\leq p<\infty$ 
 the class $S_p(\Hcal)$ becomes a Banach space endowed by the norm $\|T\|_{S_p}$. 
 If $p=\infty$ we define $S_{\infty}(\Hcal)$ as the class of bounded linear operators on $\Hcal$, 
 with 
$\|T\|_{S_\infty}:=\|T\|_{op}$, the operator norm.  For the Schatten class $S_2$ we will sometimes write $\|T\|_{\HS}$ instead of $\|T\|_{S_2}$. In the case $0<p<1$  the quantity $\|T\|_{S_p}$ only defines a  quasi-norm, and $S_p(\Hcal)$ is also complete. 
The space $S_1(\Hcal)$ is known as the {\em trace class} and an element of
 $S_2(\Hcal)$ is usually said to be a {\em Hilbert-Schmidt} operator. 
 For the  basic theory of Schatten classes we refer the reader to 
 \cite{gokr}, \cite{r-s:mp}, \cite{sim:trace}, \cite{sch:id}. 

It is well known that the class $S_2(L^2)$ is characterised by the   square integrability of the corresponding integral kernels, however, kernel estimates of this type are not effective
for classes $S_p(L^2)$ with $p<2$. This is explained by a classical Carleman's example
\cite{car:ex} on the summability of Fourier coefficients of continuous functions 
(see \cite{dr13a:nuclp} for a complete explanation of this fact). 
This obstruction explains the relevance of symbolic Schatten criteria and here we will clarify 
the advantage of the symbol approach with respect to this obstruction.
With this approach, no regularity of the kernel needs to be assumed.

We introduce $\ell^{p}$-style norms on the space of symbols $\Sigma$,
yielding discrete spaces $\ell^{p}(\Sigma)$ for $0<p\leq\infty$, normed
for $p\geq 1$.
Denoting by $\sigma_{T}$ the matrix symbol of an invariant operator $T$ provided by
Theorem \ref{THM:inv}, Schatten classes of invariant operators 
on $L^{2}(M)$ can be characterised in terms of symbols.  
Here, the condition that $T$ is invariant will mean that $T$ is strongly 
commuting with $E$ (see Theorem \ref{THM:inv}). 
On the level of the Fourier transform this means that
$$\widehat{Tf}(\ell)=\sigma(\ell) \widehat{f}(\ell)$$ for a family of
matrices $\sigma(\ell)$, i.e. $T$ assumes the familiar form of a 
Fourier multiplier.

In Section \ref{SEC:abstract} in 
Theorem \ref{THM:inv-rem} we discuss the abstract notion of symbol for
operators densely defined in a general Hilbert space $\Hcal$, and give several alternative
formulations for invariant operators, or for Fourier multipliers, relative to a fixed
partition of $\Hcal$ into a direct sum of finite dimensional subspaces,
$$\Hcal=\bigoplus_{j} H_{j}.$$
Consequently, in Theorem \ref{L2-abstract} 
we give the necessary and sufficient condition for the bounded extendability
of an invariant operator to ${\mathscr L}(\Hcal)$ in terms of its symbol, and in 
Theorem \ref{schchr-abstract} the necessary and sufficient condition for the
operator to be in Schatten classes $S_{r}(\Hcal)$ for $0<r<\infty$, as well
as the trace formula for operators in the trace class $S_{1}(\Hcal)$ in terms
of their symbols. 
As our subsequent analysis relies to a large extent on properties of
elliptic pseudo-differential operators on $M$, in Sections
\ref{SEC:Fourier} and \ref{SEC:invariant} we specify this abstract analysis
to the setting of operators densely defined on $L^{2}(M)$. The main difference
is that we now adopt the Fourier analysis to a fixed elliptic positive
pseudo-differential operator $E$ on $M$, contrary to the case of 
an operator $\En\in {\mathscr L}(\Hcal)$ in Theorem \ref{THM:inv-rem2}.

The notion of invariance depends on the choice of the spaces $H_{j}$.
Thus, in the analysis of operators on $M$ we take $H_{j}$'s to be 
the eigenspaces of $E$. However, other choices are possible. For example,
for $\Hcal=L^{2}(G)$ for a compact Lie group $G$, choosing $H_{j}$'s as
linear spans of representation coefficients for inequivalent irreducible
unitary representations of $G$, we make a link to the
quantization of pseudo-differential operators on compact Lie groups
as in \cite{rt:book}. These two partitions coincide when inequivalent representations of
$G$ produce distinct eigenvalues of the Laplacian; for example, this is the case
for $G={\rm SO(3)}.$ However, the partitions are different when inequivalent
representations produce equal eigenvalues, which is the case, for example, for 
$G={\rm SO(4)}.$ For the more explicit example on $\Hcal=L^{2}(\Tn)$ on the
torus see Remark \ref{REM:torus}. A similar choice could be made in other
settings producing a discrete spectrum and finite dimensional eigenspaces,
for example for operators in Shubin classes on $\Rn$, see Chodosh 
\cite{Chodosh} for the case $n=1$.

As an illustration we give an application to the spectral theory. The analogous concept to Schatten classes in the setting of Banach spaces is
the notion of $r$-nuclearity introduced by Grothendieck \cite{gro:me}.
It has applications to questions of the distribution of eigenvalues of operators
in Banach spaces. 
In the setting of compact Lie groups these applications have been 
discussed in \cite{dr13a:nuclp} and they include conclusions on the 
distribution or summability of eigenvalues of operators 
acting on $L^{p}$-spaces. Another application is the 
Grothendieck-Lidskii formula which is the formula for the trace of operators
on $L^{p}(M)$.
Once we have $r$-nuclearity, most of further arguments are then purely 
functional analytic, so they apply equally well
in the present setting of closed manifolds. 

The paper is organised as follows. 
In Section \ref{SEC:abstract} we discuss Fourier multipliers and their symbols in
general Hilbert spaces.
In Section \ref{SEC:Fourier} we associate a
global Fourier analysis to an elliptic positive pseudo-differential operator $E$ on 
a closed manifold $M$. In Section \ref{SEC:invariant} we introduce the class
of operators invariant relative to $E$ as well as their matrix-valued symbols,
and apply this to characterise invariant operators in Schatten classes in
Section \ref{SEC:Schatten-mfds}.

Throughout the paper, we denote $\N_{0}=\N\cup\{0\}$.
Also $\delta_{j\ell}$ will denote the Kronecker delta, i.e.
$\delta_{j\ell}=1$ for $j=\ell$, and $\delta_{j\ell}=0$ for $j\not=\ell$.

The authors would like to thank V\'eronique Fischer, Alexandre Kirilov, and Augusto Almeida de Moraes Wagner for comments.

\section{Fourier multipliers in Hilbert spaces}
\label{SEC:abstract}

In this section we present an abstract set up to describe what we will
call invariant operators, or Fourier multipliers, acting on a general
Hilbert space $\Hcal$. We will give several characterisations of such
operators and their symbols. Consequently, we will apply these notions
to describe several properties of the operators, in particular, their
boundedness on $\Hcal$ as well as the Schatten properties.

We note that direct integrals (sums in our case) of Hilbert spaces
have been investigated in a much greater generality, see e.g.
Bruhat \cite{Bruhat:BK-Tata-1968},
Dixmier \cite[Ch 2., \S 2]{Dixmier:bk-algebras},
\cite[Appendix]{Dixmier:bk-Cstar-algebras}. The setting required for our
analysis is much simpler, so we prefer to adapt it specifically for
consequent applications.

The main application of the constructions below will be in the setting when
$M$ is a compact manifold without boundary, $\Hcal=L^{2}(M)$ and
$\Hcal^{\infty}=C^{\infty}(M)$, which will be described in detail
in Section \ref{SEC:Fourier}. However, several facts can be more
clearly interpreted in the setting of abstract Hilbert spaces, which will
be our set up in this section. With this particular example in mind,
in the following theorem, we can
think of $\{e_{j}^{k}\}$ being an orthonormal basis given by eigenfunctions
of an elliptic operator on $M$, and $d_{j}$ the corresponding 
multiplicities. However, we allow flexibility in grouping the eigenfunctions
in order to be able to also cover the case of operators
on compact Lie groups.

\begin{thm}\label{THM:inv-rem}
Let $\Hcal$ be a complex Hilbert space and let $\Hcal^{\infty}\subset \Hcal$ be a dense
linear subspace of $\Hcal$. Let $\{d_{j}\}_{j\in\N_{0}}\subset\N$ and let
$\{e_{j}^{k}\}_{j\in\N_{0}, 1\leq k\leq d_{j}}$ be an
orthonormal basis of $\Hcal$ such that
$e_{j}^{k}\in \Hcal^{\infty}$ for all $j$ and $k$. Let $H_{j}:={\rm span} \{e_{j}^{k}\}_{k=1}^{d_{j}}$,
and let $P_{j}:\Hcal\to H_{j}$ be the orthogonal projection.
For $f\in\Hcal$, we denote $$\widehat{f}(j,k):=(f,e_{j}^{k})_{\Hcal}$$ and let
$\widehat{f}(j)\in \ce^{d_{j}}$ denote the column of $\widehat{f}(j,k)$, $1\leq k\leq d_{j}.$
Let $T:\Hcal^{\infty}\to \Hcal$ be a linear operator.
Then the following
conditions are equivalent:
\begin{itemize}
\item[(A)] For each $j\in\ene_0$, we have $T(H_j)\subset H_j$. 
\item[(B)] For each $\ell\in\ene_0$ there exists a matrix 
$\sigma(\ell)\in\ce^{d_{\ell}\times d_{\ell}}$ such that for all $e_j^k$ 
$$
\widehat{Te_j^k}(\ell,m)=\sigma(\ell)_{mk}\delta_{j\ell}.
$$
\item[(C)]  If in addition, $e_j^k$ are in the domain of $T^*$ for all $j$ and $k$, then 
for each $\ell\in\ene_0 $ there exists a matrix 
$\sigma(\ell)\in\ce^{d_{\ell}\times d_{\ell}}$ such that
 \[\widehat{Tf}(\ell)=\sigma(\ell)\widehat{f}(\ell)\]
 for all $f\in\Hcal^{\infty}.$
\end{itemize}

The matrices $\sigma(\ell)$ in {\rm (B)} and {\rm (C)} coincide.

The equivalent properties {\rm (A)--(C)} follow from the condition 
\begin{itemize}
\item[(D)] For each $j\in\ene_0$, we have
$TP_j=P_jT$ on $\Hcal^{\infty}$.
\end{itemize}
If, in addition, $T$ extends to a bounded operator
$T\in{\mathscr L}(\Hcal)$ then {\rm (D)} is equivalent to {\rm (A)--(C)}.
\end{thm} 

Under the assumptions of Theorem \ref{THM:inv-rem}, we have the direct sum 
decomposition
\begin{equation}\label{EQ:sum}
\Hcal = \bigoplus_{j=0}^{\infty} H_{j},\quad H_{j}={\rm span} \{e_{j}^{k}\}_{k=1}^{d_{j}},
\end{equation}
and we have $d_{j}=\dim H_{j}.$
The two applications that we will consider will be with $\Hcal=L^{2}(M)$ for a
compact manifold $M$ with $H_{j}$ being the eigenspaces of an elliptic 
pseudo-differential operator $E$, or with $\Hcal=L^{2}(G)$ for a compact Lie group
$G$ with $$H_{j}=\textrm{span}\{\xi_{km}\}_{1\leq k,m\leq d_{\xi}}$$ for a
unitary irreducible representation $\xi\in[\xi_{j}]\in\widehat{G}$. The difference
is that in the first case we will have that the eigenvalues of $E$ corresponding
to $H_{j}$'s are all distinct, while in the second case the eigenvalues of the Laplacian
on $G$ for which $H_{j}$'s are the eigenspaces, may coincide.
In Remark \ref{REM:torus} we give an example of this difference for operators on
the torus $\Tn$.

In view of properties (A) and (C), respectively, an operator $T$ satisfying any of
the equivalent properties (A)--(C) in
Theorem \ref{THM:inv-rem}, will be called an {\em invariant operator}, or
a {\em Fourier multiplier relative to the decomposition
$\{H_{j}\}_{j\in\N_{0}}$} in \eqref{EQ:sum}.
If the collection $\{H_{j}\}_{j\in\N_{0}}$
is fixed once and for all, we can just say that $T$ is {\em invariant}
or a {\em Fourier multiplier}.

The family of matrices $\sigma$ will be
called the {\em matrix symbol of $T$ relative to the partition $\{H_{j}\}$ and to the
basis $\{e_{j}^{k}\}$.}
It is an element of the space $\Sigma$ defined by
\begin{equation}\label{EQ:Sigma1}
\Sigma=\{\sigma:\N_{0}\ni\ell\mapsto\sigma(\ell)\in \ce^{d_{\ell}\times d_{\ell}}\}.
\end{equation}
A criterion for the extendability of $T$ to ${\mathscr L}(\Hcal)$ in terms
of its symbol will be given in Theorem \ref{L2-abstract}.

For $f\in\Hcal$, in the notation of Theorem \ref{THM:inv-rem},
by definition we have
\begin{equation}\label{EQ:ser}
f=\sum_{j=0}^{\infty} \sum_{k=1}^{d_{j}} \widehat{f}(j,k) e_{j}^{k}
\end{equation}
with the convergence of the series in $\Hcal$.
Since $\{e^k_j\}_{j\geq 0}^{1\leq k\leq d_j}$ is a complete orthonormal 
system on  $\Hcal$, for all $f\in \Hcal$ we have the Plancherel formula
\beq \label{EQ:Plancherel}
\|f\|^2_{\Hcal}=\sum\limits_{j=0}^{\infty}\sum\limits_{k=1}^{d_j}|( f,e_j^k)|^2
=  \sum\limits_{j=0}^{\infty}\sum\limits_{k=1}^{d_j}|\widehat{f}(j,k)|^{2}
=\|\widehat{f}\|^{2}_{\ell^2(\N_{0},\Sigma)},
\eq
where we interpret $\widehat{f}\in\Sigma$ as an element of the space
\begin{equation}\label{EQ:aux3}
\ell^2(\N_{0,}\Sigma)=
\{h:\ene_0\rightarrow \prod\limits_d\ce^{d}: h(j)\in \ce^{d_j}\, \mbox{ and }\,\sum\limits_{j=0}^{\infty}\sum\limits_{k=1}^{d_j}|h(j,k)|^2<\infty\}, 
\end{equation}
and where we have written $h(j,k)=h(j)_k$. 
In other words, $\ell^2(\N_{0,}\Sigma)$ is the space of all $h\in\Sigma$ such that
$$
\sum\limits_{j=0}^{\infty}\sum\limits_{k=1}^{d_j}|h(j,k)|^2<\infty.
$$
We endow  $\ell^2(\N_{0},\Sigma)$ with the norm
\begin{equation}\label{EQ:aux4}
\|h\|_{\ell^2(\N_{0,}\Sigma)}:=\left(\sum\limits_{j=0}^{\infty}\sum\limits_{k=1}^{d_j}|h(j,k)|^2\right)^{\half}.
\end{equation}

We note that the matrix symbol $\sigma(\ell)$ depends 
not only on the partition \eqref{EQ:sum} but also
on the choice of the orthonormal basis.
Whenever necessary, we will indicate the dependance of $\sigma$ on the orthonormal 
basis by writing $(\sigma,\{e_j^k\}_{j\geq 0}^{1\leq k\leq d_j} )$ and we also will refer to 
$(\sigma,\{e_j^k\}_{j\geq 0}^{1\leq k\leq d_j} )$ as the {\em symbol} of $T$. 
Throughout this  section the orthonormal basis will be fixed and unless there is some 
risk of confusion the symbols will be denoted simply by $\sigma$.  
In the invariant language,  
we have that the transpose of the symbol,
$\sigma(j)^{\top}=T|_{H_{j}}$ is just the restriction of
$T$ to $H_{j}$, which is well defined in view of the property (A).

We will also sometimes 
write $T_{\sigma}$ to indicate that $T_{\sigma}$ is an operator corresponding to the 
symbol $\sigma $. It is clear from the definition that invariant operators are 
uniquely determined by their symbols. Indeed, if $T=0$ we obtain 
$\sigma=0$  for any choice of an orthonormal basis.  
Moreover, we note that by taking $j=\ell$ in (B) of Theorem \ref{THM:inv-rem} we obtain
the formula for the symbol:
\beq\label{symbinv}
\sigma(j)_{mk}=\widehat{Te_j^k}(j,m),
\eq
for all $1\leq k,m\leq d_j$. The formula (\ref{symbinv}) furnishes 
an explicit formula for the symbol in terms of the operator and the orthonormal basis. 
The definition of Fourier coefficients tells us that for invariant operators
we have 
\beq\label{symbinv2}
\sigma(j)_{mk}=({Te_j^k},e_j^m)_{H}.
\eq
In particular,  for the identity operator $T=I$ we have $\sigma_{I}(j)=I_{d_{j}}$,
where $I_{d_{j}}\in \C^{{d_{j}}\times {d_{j}}}$ is the identity matrix.

Let us now indicate a formula relating symbols 
with respect to different orthonormal basis. 
If $\{e_{\alpha}\}$ and $\{f_{\alpha}\}$ are orthonormal bases of 
$\Hcal$, we consider the unitary operator $U$ determined by 
$U(e_{\alpha})=f_{\alpha}$. Then we have
\[
(Te_{\alpha}, e_{\beta})_{\Hcal}=(UTe_{\alpha}, Ue_{\beta})_{\Hcal}
=(UTU^*Ue_{\alpha}, Ue_{\beta})_{\Hcal}
=(UTU^*f_{\alpha}, f_{\beta})_{\Hcal}.
\]
Thus, if $(\sigma_{T}, \{e_{\alpha}\})$ denotes the symbol of $T$ with respect to the 
orthonormal basis $\{e_{\alpha}\}$ and $(\sigma_{UTU^*}, \{f_{\alpha}\})$ 
denotes the symbol of $UTU^*$ with respect to the orthonormal basis $\{f_{\alpha}\}$ 
we have obtained the relation
\beq\label{difsymb} (\sigma_{T}, \{e_{\alpha}\})=({\sigma_{UTU^*}}, \{f_{\alpha}\}).\eq
Thus, the equivalence relation of basis $\{e_{\alpha}\}\sim  \{f_{\alpha}\}$ given by 
a unitary operator $U$ induces
the equivalence relation on the set $\Sigma$ of symbols given by 
\eqref{difsymb}. In view of this,
we can also think of the symbol being independent of a choice of basis, as an element of the space
$\Sigma/\sim$ with the equivalence relation given by
\eqref{difsymb}.

We make another remark concerning part (C) of Theorem \ref{THM:inv-rem}.
We use the condition that $e_j^k$ are in the domain ${\rm Dom}(T^*)$ of $T^*$ in showing the implication
(B) $\Longrightarrow$ (C). Since $e_j^k$'s give a basis in $\Hcal$, and are all contained in 
${\rm Dom}(T^*)$, it follows that ${\rm Dom}(T^*)$ is dense in $\Hcal$. In particular, 
by \cite[Theorem VIII.1]{r-s:vol1}, $T$ must be closable (in part (C)). These conditions are not restrictive for the further
analysis since they are satisfied in several natural applications.

The principal application of the notions above will be as follows,
except for in the sequel we sometimes need more general operators $E$
unbounded on $\Hcal$. In order to distinguish from this general case,
in the following theorem we use the notation $\En$.

\begin{thm}\label{THM:inv-rem2}
Continuing with the notation of Theorem \ref{THM:inv-rem}, let 
$\En\in {\mathscr L}(\Hcal)$ be a linear continuous operator such
that $H_{j}$ are its eigenspaces:
$$\En e_{j}^{k}=\lambda_{j} e_{j}^{k}$$ for each $j\in\ene_{0}$ and all $1\leq k\leq d_{j}$.
Then equivalent conditions {\rm (A)--(C)} imply the property 
\begin{itemize}
\item[(E)]
For each $j\in\ene_{0}$ and $1\leq k\leq j$, we have
$T\En e_{j}^{k}=\En T e_{j}^{k},$
\end{itemize}
and if 
$\lambda_{j}\not=\lambda_{\ell}$ 
for $j\not=\ell$, then {\rm (E)} is equivalent to  
properties {\rm (A)--(C)}.

\smallskip
Moreover, if $T$ extends to a bounded operator
$T\in{\mathscr L}(\Hcal)$ then equivalent properties {\rm (A)--(D)} 
imply the condition
\begin{itemize}
\item[(F)]  $T\En=\En T$ on $\Hcal$,
\end{itemize}
and if also $\lambda_{j}\not=\lambda_{\ell}$ 
for $j\not=\ell$, then {\rm (F)} is equivalent to  {\rm (A)--(E)}.
\end{thm} 

For an operator $T=F(\En)$, when it is well-defined by the spectral calculus, we have
\begin{equation}\label{EQ:symbol-Fa}
\sigma_{F(\En)}(j)=F(\lambda_{j}) I_{d_{j}}.
\end{equation}
In fact, this is also well-defined then 
for a function $F$ defined on $\lambda_j$, with finite values which 
are e.g. $j$-uniformly bounded (also for non self-adjoint $E_o$).

We have the following criterion for the extendability of
a densely defined invariant operator $T:\Hcal^{\infty}\to \Hcal$ to ${\mathscr L}(\Hcal)$,
which was an additional hypothesis for properties (D) and (F).
In the statements below we fix a partition into $H_{j}$'s as in \eqref{EQ:sum}
and the invariance refers to it.

\begin{thm}\label{L2-abstract} 
An invariant linear operator $T:\Hcal^{\infty}\to \Hcal$ extends to a bounded
operator from $\Hcal$ to $\Hcal$ if and only if its symbol $\sigma$ satisfies
$\sup\limits_{\ell\in\N_{0}}\|\sigma(\ell)\|_{{\mathscr L}(H_{\ell})}<\infty.$
Moreover, denoting this extension also by $T$, we have
\[
\| T\|_{{\mathscr L}(\Hcal)} =\sup\limits_{\ell\in \N_{0}}\|\sigma(\ell)\|_{{\mathscr L}(H_{\ell})}.
\]
\end{thm} 

We also record the formula for the symbol of the composition of two invariant
operators:

\begin{prop}\label{comp1-abstract} 
If $S,T:\Hcal^{\infty}\to \Hcal$ are invariant operators with respect to the same 
orthonormal partition, and such that the domain of $S\circ T$ contains $\Hcal^{\infty}$,
then $S\circ T:\Hcal^{\infty}\to \Hcal$ is also 
invariant with respect to the same partition. Moreover, if $\sigma_S$ denotes the symbol of 
$S$ and $\sigma_T$ denotes the symbols of $T$ 
with respect to  the same orthonormal basis then 
\[\sigma_{S\circ T}=\sigma_S\sigma_T,\]
i.e. $\sigma_{S\circ T}(j)=\sigma_S(j)\sigma_T(j)$ for all $j\in\N_{0}.$
\end{prop} 

We now show another application of 
the above notions to give a characterisation of Schatten classes of invariant operators in
terms of their symbols. 

\begin{thm}\label{schchr-abstract} Let $0<r<\infty$. 
An invariant operator $T\in {\mathscr L}(\Hcal)$ with symbol 
$\sigma$ is in the Schatten class
$S_r(\Hcal)$ 
if and only if  $$\sum\limits_{\ell=0}^{\infty}\|\sigma(\ell)\|_{S_r(H_{\ell})}^r<\infty.$$ 
Moreover
\begin{equation}\label{EQ:Sch1}
\|T\|_{S_r(\Hcal)}=\p{\sum\limits_{\ell=0}^{\infty}\|\sigma(\ell)\|_{S_r(H_{\ell})}^r}^{1/r}.
\end{equation}
In particular, if $T$ is in the trace class
$S_1(\Hcal)$, then we have the trace formula
\begin{equation}\label{EQ:Sch2}
\Tr(T)=\sum\limits_{\ell=0}^{\infty}\Tr(\sigma(\ell)).
\end{equation}
\end{thm}

\begin{rem}\label{REM:torus}
We note that the membership in ${\mathscr L}(\Hcal)$ and in the Schatten classes
$S_{r}(\Hcal)$ does not depend on the decomposition of $\Hcal$ into subspaces $H_{j}$
as in \eqref{EQ:sum}. However, the notion of invariance does depend on it.
For example, let $\Hcal=L^{2}(\Tn)$ for the $n$-torus $\Tn=\Rn/\Zn$. 
Choosing 
$$H_{j}={\rm span}\{ e^{2\pi{\rm  i} j\cdot x} \}, \quad j\in\Zn,$$ 
we recover the classical notion of invariance on compact Lie groups
and moreover, invariant operators
with respect to $\{H_{j}\}_{j\in\Zn}$ are the translation invariant operators on the torus
$\Tn$. However, to recover the construction of Section \ref{SEC:invariant}
on manifolds, we take
$\widetilde{H_{\ell}}$ to be the eigenspaces of the Laplacian $E$ on $\Tn$, so that
$$
\widetilde{H_{\ell}}=\bigoplus_{|j|^{2}=\ell} H_{j}=
{\rm span}\{e^{2\pi{\rm  i} j\cdot x}:\; j\in\Zn \textrm{ and }
|j|^{2}=\ell\},
\quad \ell\in\N_{0}.
$$ 
Then translation invariant operators on $\Tn$, i.e. operators invariant
relative to the partition $\{H_{j}\}_{j\in\Zn}$, are also invariant relative to the partition
$\{\widetilde{H_{\ell}}\}_{\ell\in\N_{0}}$ 
(or relative to the Laplacian, in terminology of Section \ref{SEC:invariant}).

If we have information on the eigenvalues of
$E$, like we do on the torus, we may sometimes also
recover invariant operators relative to the partition
$\{\widetilde{H_{\ell}}\}_{\ell\in\N_{0}}$ as linear combinations of translation invariant
operators composed with phase shifts and complex conjugation.
\end{rem}

\section{Fourier analysis associated to an elliptic operator}
\label{SEC:Fourier}

One of the main applications of the described setting is to study operators on compact manifolds, so we start this
section by describing the discrete Fourier analysis associated to an 
elliptic positive pseudo-differential operator as an adaptation of the construction in
Section \ref{SEC:abstract}. In order to fix the further notation  
we give some explicit expressions for notions of Section \ref{SEC:abstract}
in this setting.

Let $M$ be a compact smooth manifold of dimension $n$ without boundary, endowed with a fixed volume $dx$.  
 We denote by $\Psi^{\nu}(M)$ the H\"ormander class of pseudo-differential 
 operators of order $\nu\in\er$,
 i.e. operators which, in every coordinate chart, are operators in H\"ormander classes 
 on $\Rn$ with symbols
 in $S^\nu_{1,0}$, see e.g. \cite{shubin:r} or \cite{rt:book}.
 For simplicity we may be using the class  $\Psi^{\nu}_{cl}(M)$ of classical operators, i.e. operators
 with symbols having (in all local coordinates) an asymptotic expansion of the symbol in
 positively homogeneous components (see e.g. \cite{Duis:BK-FIO-2011}).
 Furthermore, we denote by $\Psi_{+}^{\nu}(M)$ the class of positive definite operators in 
 $\Psi^{\nu}_{cl}(M)$,
 and by $\Psi_{e}^{\nu}(M)$ the class of elliptic operators in $\Psi^{\nu}_{cl}(M)$. Finally, 
 $$\Psi_{+e}^{\nu}(M):=\Psi_{+}^{\nu}(M)\cap \Psi_{e}^{\nu}(M)$$ 
 will denote the  class of classical positive elliptic 
 pseudo-differential operators of order $\nu$.
 We note that complex powers of such operators are well-defined, see e.g.
 Seeley \cite{Seeley:complex-powers-1967}. 
 In fact, all pseudo-differential operators considered from now on will be classical, so we may omit explicitly mentioning it every time, but we note
 that we could equally work with general operators in $\Psi^{\nu}(M)$ since their
 powers have similar properties, see e.g. \cite{Strichartz:functional-calculus-AJM-1972}.
 
We now associate a discrete Fourier analysis to the operator 
$E\in\Psi_{+e}^{\nu}(M)$ inspired 
by those constructions
considered by Seeley (\cite{see:ex}, \cite{see:exp}), see also 
Greenfield and Wallach \cite{Greenfield-Wallach:hypo-TAMS-1973}. 
However, we adapt it to our purposes and in the sequel also
indicate several auxiliary statements concerning the
eigenvalues of $E$ and their multiplicities, useful to us in the subsequent analysis.
In general, the construction below is exactly the one appearing in
Theorem \ref{THM:inv-rem} with a particular choice of a partition. 

The eigenvalues of $E$ (counted without multiplicities)
form a sequence $\{\lambda_j\}$ which we order so that
\begin{equation}\label{EQ:lambdas}
0=\lambda_{0}<\lambda_{1}<\lambda_{2}<\cdots.
\end{equation}
For each eigenvalue $\lambda_j$, there is
the corresponding finite dimensional eigenspace $H_j$ of functions on $M$, which are smooth due to the 
ellipticity of $E$. We set 
$$
d_j:=\dim H_j, 
\textrm{ and } H_0:=\ker E, \; \lambda_0:=0.
$$
We also set $d_{0}:=\dim H_{0}$. Since the operator $E$ is elliptic, it is Fredholm,
hence also $d_{0}<\infty$ (we can refer to \cite{Atiyah:global-aspects-1968}, \cite{ho:apde2} for
various properties of $H_{0}$ and $d_{0}$).

We fix  an orthonormal basis of $L^2(M)$ consisting of eigenfunctions of $E$:
\beq\label{fam}\{e^k_j\}_{j\geq 0}^{1\leq k\leq d_j},\eq 
where $\{e^k_j\}^{1\leq k\leq d_j}$ is an orthonormal basis of $H_j$. 
Let $P_j:L^2(M)\rightarrow H_j$ be the corresponding projection. 
We shall denote by $(\cdot,\cdot)$ the inner product of $L^2(M)$. 
 We observe that we have
 \[P_jf=\sum\limits_{k=1}^{d_j}(f,e_j^k) e_j^k,\]
for $f\in L^2(M)$. The `Fourier' series takes the form 
\[f=\sum\limits_{j=0}^{\infty}\sum\limits_{k=1}^{d_j}(f,e_j^k )e_j^k,\]
for each $f\in L^2(M)$.
The Fourier coefficients of $f\in L^2(M)$ with respect to the orthonormal basis $\{e^k_j\}$ 
will be denoted by 
\begin{equation}\label{EQ:F-coeff}
(\efee f)(j,k):=\widehat{f}(j,k):=(f,e_j^k).
\end{equation}
We will call the collection of $\widehat{f}(j,k)$ the {\em Fourier coefficients of $f$ relative to $E$},
or simply the {\em Fourier coefficients of $f$}.

\smallskip
Since $\{e^k_j\}_{j\geq 0}^{1\leq k\leq d_j}$ forms a complete orthonormal 
system in  $L^2(M)$, for all $f\in L^2(M)$ we have the Plancherel formula
\eqref{EQ:Plancherel}, namely,
\beq \label{EQ:Plancherel2}
\|f\|^2_{L^{2}(M)}=\sum\limits_{j=0}^{\infty}\sum\limits_{k=1}^{d_j}|( f,e_j^k)|^2
=  \sum\limits_{j=0}^{\infty}\sum\limits_{k=1}^{d_j}|\widehat{f}(j,k)|^{2}
=\|\widehat{f}\|^{2}_{\ell^2(\ene_0,\Sigma)},
\eq
where the space $\ell^2(\ene_0,\Sigma)$ and its norm are as in
\eqref{EQ:aux3} and \eqref{EQ:aux4}.

We can think of $\efee=\efee_{M}$ as of the Fourier transform being an isometry
from $L^2(M)$ into 
$\ell^2(\ene_0,\Sigma)$. The inverse of this Fourier transform can be then expressed by 
\beq\label{Fourier1inv}
(\efee^{-1}{h})(x)=\sum\limits_{j= 0}^{\infty}\sum\limits_{k=1}^{d_j}h(j,k)e_j^k(x).
\eq
If $f\in L^2(M)$, we also write
\[\widehat{f}(j)=\left(\begin{array}{c}\widehat{f}(j,1)\\
\vdots\\
\widehat{f}(j,d_j)\end{array} \right)\in\ce^{d_j},\]
thus thinking of the Fourier transform always as a column vector.
In particular, we think of 
\[\widehat{e_j^k}(\ell)=\left(\widehat{e_j^k}(\ell,m)\right)_{m=1}^{d_{\ell}}\]
as of a column, and we notice that
\beq
\widehat{e_j^k}(\ell,m)=\delta_{j\ell}\delta_{km}.
\label{eqdeltas}
\eq
Smooth functions on $M$ can be characterised by
\begin{align}\label{EQ:smooth}
f\in C^{\infty}(M)  & \Longleftrightarrow 
\forall N \; \exists C_{N}: \; |\widehat{f}(j,k)|\leq C_{N} (1+\lambda_{j})^{-N}
\textrm{ for all } j, k \\ \nonumber
& \Longleftrightarrow 
\forall N \; \exists C_{N}: \; |\widehat{f}(j)|\leq C_{N} (1+\lambda_{j})^{-N}
\textrm{ for all } j,
\end{align}
where $|\widehat{f}(j)|$ is the norm of the vector $\widehat{f}(j)\in\C^{d_{j}}.$
The implication `$\Longleftarrow$' here is immediate, while `$\Longrightarrow$'
follows from the Plancherel formula \eqref{EQ:Plancherel} and the fact that
for $f\in C^{\infty}(M)$ we have $(I+E)^{N}f\in L^{2}(M)$ for any $N$.

\smallskip
For $u\in \Dcal'(M)$, we denote its Fourier coefficient 
$$\widehat{u}(j,k):=u(\overline{e_{j}^{k}}),$$ and by duality, the space of distributions
can be characterised by
$$
f\in \Dcal'(M)   \Longleftrightarrow 
\exists M \; \exists C: \; |\widehat{u}(j,k)|\leq C(1+\lambda_{j})^{M}
\textrm{ for all } j, k.
$$

We will denote by $H^{s}(M)$ the usual Sobolev space over $L^{2}$ on $M$.
This space can be defined in local coordinates or, by the fact that 
$E\in \Psi^{\nu}_{+e}(M)$ is positive and elliptic with $\nu>0$, it can be
characterised by
\begin{multline}\label{EQ:Sob-char}
 f\in H^{s}(M) \Longleftrightarrow (I+E)^{s/\nu} f\in L^{2}(M)
 \Longleftrightarrow
 \{(1+\lambda_{j})^{s/\nu}\widehat{f}(j)\}_{j} \in \ell^{2}(\N_{0},\Sigma) \\
 \Longleftrightarrow
 \sum\limits_{j=0}^{\infty}\sum\limits_{k=1}^{d_j}
 (1+\lambda_{j})^{2s/\nu}|\widehat{f}(j,k)|^{2}<\infty ,
\end{multline}
the last equivalence following from the Plancherel formula
\eqref{EQ:Plancherel}.
For the characterisation of analytic functions (on compact manifolds $M$) 
we refer to Seeley \cite{see:exp} and \cite{dr:eig}.


\section{Invariant operators and symbols on compact manifolds}
\label{SEC:invariant}

We now discuss an application of a notion of an invariant operator and of its symbol
from Theorem \ref{THM:inv-rem} in the case of 
$\Hcal=L^{2}(M)$ and $\Hcal^{\infty}=C^{\infty}(M)$ and describe its basic
properties. 
We will consider operators $T$ densely defined on $L^{2}(M)$, and we will be making
a natural assumption that their domain contains $C^{\infty}(M)$.
We also note that while in Theorem \ref{THM:inv-rem2} it was assumed that the
operator $\En$ is bounded on $\Hcal$, this is no longer the case for the operator $E$ here.
Indeed, an elliptic  pseudo-differential operator $E\in\Psi_{+e}^{\nu}(M)$ of order
$\nu>0$ is not bounded on $L^{2}(M)$. 

Moreover, we do not want to assume that $T$ extends to a bounded operator on 
$L^{2}(M)$
to obtain analogues of properties (D) and (F) in Section \ref{SEC:abstract},
because this is too restrictive from the point of view of differential operators.
Instead, we show that in the present setting it is enough to assume that
$T$ extends to a continuous operator on $\Dcal'(M)$ to reach the same
conclusions.

So, we combine the statement of Theorem \ref{THM:inv-rem} and the necessary 
modification of Theorem \ref{THM:inv-rem2} to the setting of
Section \ref{SEC:Fourier} as follows.

We also remark that Part (iv) of the following theorem provides a correct formulation for
a missing assumption in \cite[Theorem 3.1, (iv)]{Delgado-Ruzhansky:CRAS-kernels}.

\begin{thm}\label{THM:inv}
Let $M$ be a closed manifold and 
let $T:\cinfm\to L^{2}(M)$
be a linear operator.
Then the following
conditions are equivalent:
\begin{itemize}
\item[(i)] For each $j\in\ene_0$, we have $T(H_j)\subset H_j$. 
\item[(ii)]
For each $j\in\ene_{0}$ and $1\leq k\leq j$, we have
$TE e_{j}^{k}=ET e_{j}^{k}.$
\item[(iii)] For each $\ell\in\ene_0$ there exists a matrix 
$\sigma(\ell)\in\ce^{d_{\ell}\times d_{\ell}}$ such that for all $e_j^k$ 
\beq\label{invadef}\widehat{Te_j^k}(\ell,m)=\sigma(\ell)_{mk}\delta_{j\ell}.\eq
\item[(iv)]  If, in addition, the domain of $T^*$ contains $C^\infty(M)$, then for each $\ell\in\ene_0 $ there exists a matrix 
$\sigma(\ell)\in\ce^{d_{\ell}\times d_{\ell}}$ such that
 \[\widehat{Tf}(\ell)=\sigma(\ell)\widehat{f}(\ell)\]
 for all $f\in\cinfm.$
\end{itemize}
The matrices $\sigma(\ell)$ in {\rm (iii)} and {\rm (iv)} coincide.

\smallskip
If $T$
extends to a linear continuous operator
$T:\Dcal'(M)\rightarrow \Dcal'(M)$ then
the above properties are also equivalent to the following ones:
\begin{itemize}
\item[(v)] For each $j\in\ene_0$, we have
$TP_j=P_jT$ on $C^{\infty}(M)$.
\item[(vi)]  $TE=ET$ on $L^{2}(M)$.
\end{itemize}
\end{thm} 

If any of the equivalent conditions (i)--(iv) of Theorem \ref{THM:inv} are satisfied, 
we say that the operator $T:\cinfm\rightarrow L^{2}(M)$ 
is {\em invariant (or is a Fourier multiplier) relative to $E$}.
We can also say 
that $T$ is $E$-invariant or is an $E$-multiplier. 
This recovers the notion of invariant operators 
given by Theorem \ref{THM:inv-rem}, with respect to the partitions
$H_{j}$'s in \eqref{EQ:sum} which are fixed being the eigenspaces of $E$.
When there is no risk of confusion we will just refer to such kind of operators 
as invariant operators or as Fourier multipliers.
It is clear from (i) that the operator $E$ itself or functions of $E$ defined
by the functional calculus are invariant relative to $E$. 

We note that the boundedness of $T$ on $L^{2}(M)$ needed for conditions
(D) and (F) in Theorem \ref{THM:inv-rem} and in Theorem \ref{THM:inv-rem2}
is now replaced by the condition that $T$ is continuous on $\Dcal'(M)$ which
explored the additional structure of $L^{2}(M)$ and allows application to
differential operators.

We call $\sigma$ in (iii) and (iv) the {\em matrix symbol of $T$} or simply 
the {\em symbol}. It is an element of the space $\Sigma=\Sigma_{M}$ defined by
\begin{equation}\label{EQ:Sigma}
\Sigma_M:=\{\sigma:\N_{0}\ni\ell\mapsto\sigma(\ell)\in \ce^{d_{\ell}\times d_{\ell}}\}.
\end{equation}
Since the expression for the symbol depends only on the basis $e_{j}^{k}$ and not
on the operator $E$ itself, this notion coincides with the symbol defined in
Theorem \ref{THM:inv-rem}.

Let us comment on several conditions in Theorem \ref{THM:inv} in this setting.
Assumptions (v) and (vi) are stronger than those in (i)--(iv).
On one hand, clearly (vi) contains (ii). On the other hand, 
it can be shown that assumption (v) implies (i) without the 
additional hypothesis that $T$ is continuous on $\Dcal'(M)$.

In analogy to the strong commutativity in (v), {\em if $T$ is continuous on $\Dcal'(M)$},
so that all the assumptions (i)--(vi) are equivalent, we may say that
$T$ is {\em strongly invariant relative to $E$} in this case.

The expressions in (vi) make sense as both sides are defined (and even continuous) on
$\Dcal'(M)$. 

We also note that without additional assumptions, it is known from the
general theory of densily defined operators on Hilbert spaces that conditions
(v) and (vi) are generally not equivalent, see e.g.
Reed and Simon \cite[Section VIII.5]{r-s:vol1}.
If $T$ is a differential operator, the additional assumption of continuity
on $\Dcal'(M)$ for 
parts (v) and (vi) is satisfied.  In
\cite[Section 1, Definition 1]{Greenfield-Wallach:hypo-TAMS-1973} 
Greenfield and Wallach
called a differential operator $D$ to be an $E$-invariant operator if  $ED=DE$,
which is our condition (vi). However, Theorem  \ref{THM:inv} describes more
general operators as well as reformulates them in the form of Fourier multipliers
that will be explored in the sequel.

There will be several useful classes of symbols, in particular the moderate growth
class
\begin{equation}\label{EQ:sym-tempe}
{\mathcal S}'(\Sigma):=\{\sigma\in\Sigma:
\exists N, C \textrm{ such that }
\|\sigma(\ell)\|_{op}\leq C(1+\lambda_{\ell})^{N} \; \forall \ell\in\N_{0}
\},
\end{equation}
where $$ \|\sigma(\ell)\|_{op}=\|\sigma(\ell)\|_{{\mathscr L}(H_{\ell})}$$ 
denotes the matrix multiplication
operator norm with respect to $\ell^2(\ce^{d_{\ell}})$.

In the case when $M$ is a compact Lie group and $E$ is a Laplacian on $G$,
left-invariant operators on $G$, i.e. operators commuting with the left action of $G$,
are also invariant relative to $E$ in the sense of Theorem \ref{THM:inv}.
 However, we need an adaptation of the above construction
since the natural decomposition into $H_{j}$'s in \eqref{EQ:sum} may in general violate the condition \eqref{EQ:lambdas}.

As in Section \ref{SEC:abstract} since the notion of the symbol depends only
on the basis, for the identity operator $T=I$ we have $$\sigma_{I}(j)=I_{d_{j}},$$
where $I_{d_{j}}\in \C^{I_{d_{j}}\times I_{d_{j}}}$ is the identity matrix, and 
for an operator $T=F(E)$, when it is well-defined by the spectral calculus, we have
\begin{equation}\label{EQ:symbol-F}
\sigma_{F(E)}(j)=F(\lambda_{j}) I_{d_{j}}.
\end{equation}

We now discuss how invariant operators can be expressed in terms of their symbols.

\begin{prop}\label{PROP:quant}
An invariant operator $T_{\sigma}$ associated to the symbol $\sigma$ can be written 
in the following way:
\begin{align}
T_{\sigma}f(x)=&\sum\limits_{\ell=0}^{\infty}\sum\limits_{m=1}^{d_{\ell}}(\sigma(\ell)\widehat{f}(\ell))_me_{\ell}^m(x)\label{form23}\\
=&\sum\limits_{\ell=0}^{\infty}[\sigma(\ell)\widehat{f}(\ell)]^{\top} e_{\ell}(x),\nonumber
\end{align}
where $[\sigma(\ell)\widehat{f}(\ell)]$ denotes the column-vector, and 
$[\sigma(\ell)\widehat{f}(\ell)]^{\top}e_{\ell}(x)$ denotes the multiplication
(the scalar product)
of the column-vector $[\sigma(\ell)\widehat{f}(\ell)]$ with the column-vector
$e_{\ell}(x)=(e_{\ell}^{1}(x),\cdots, e_{\ell}^{m}(x))^{\top}$.
In particular, we also have
\begin{equation}\label{EQ:Tsigma-e}
(T_{\sigma}e_{j}^{k})(x)=\sum\limits_{m=1}^{d_{j}} \sigma(j)_{mk}e_{j}^{m}(x).
\end{equation}
If $\sigma\in {\mathcal S}'(\Sigma)$ and $f\in C^{\infty}(M)$, the convergence in
\eqref{form23} is uniform.
\end{prop}
\begin{proof} 
Formula \eqref{form23} follows from Part (iv) of Theorem \ref{THM:inv},
with uniform convergence for $f\in C^{\infty}(M)$ in view of
\eqref{EQ:sym-tempe}.
Then, using \eqref{form23} and (\ref{eqdeltas}) we can calculate
\begin{align*}
(T_{\sigma}e_{j}^{k})(x)=&\sum\limits_{\ell=0}^{\infty}\sum\limits_{m=1}^{d_{\ell}}(\sigma(\ell)\widehat{e_{j}^{k}}(\ell))_me_{\ell}^m(x)\\ 
=&\sum\limits_{\ell=0}^{\infty}\sum\limits_{m=1}^{d_{\ell}}\left(\sum\limits_{i=1}^{d_{\ell}}(\sigma(\ell))_{mi}\widehat{e_{j}^{k}}(\ell,i)\right)e_{\ell}^m(x)\\
=&\sum\limits_{\ell=0}^{\infty}\sum\limits_{m=1}^{d_{\ell}}\sum\limits_{i=1}^{d_{\ell}}(\sigma(\ell))_{mi}\delta_{j\ell}\delta_{ki}e_{\ell}^m(x)\\
=&\sum\limits_{m=1}^{d_{j}}(\sigma(j))_{mk}e_{j}^m(x),
\end{align*}
yielding \eqref{EQ:Tsigma-e}.
\end{proof}

Theorem \ref{L2-abstract} characterising invariant operators bounded on 
$L^{2}(M)$ now becomes
\begin{thm}\label{L2} 
An invariant linear operator $T:\cinfm\rightarrow L^{2}(M)$ extends to a bounded
operator from $L^{2}(M)$ to $L^{2}(M)$ if and only if its symbol $\sigma$ satisfies
$$\sup\limits_{\ell\in\N_{0}}\|\sigma(\ell)\|_{op}<\infty,$$
where $ \|\sigma(\ell)\|_{op}=\|\sigma(\ell)\|_{{\mathscr L}(H_{\ell})}$ 
is the matrix multiplication
operator norm with respect to $H_{\ell}\simeq\ell^2(\ce^{d_{\ell}})$.
Moreover, we have
\[
\| T\|_{{\mathscr L}(L^{2}(M))} =\sup\limits_{\ell\in \N_{0}}\|\sigma(\ell)\|_{op}.
\]
\end{thm} 

This can be extended to Sobolev spaces. We will use 
the multiplication property for Fourier multipliers which is a direct
consequence of Proposition \ref{comp1-abstract}:

\begin{prop}\label{comp1} 
If $S,T:C^{\infty}(M)\to L^{2}(M)$ are invariant operators with respect to $E$ 
such that the domain of $S\circ T$ contains $C^{\infty}(M)$,
then $S\circ T:C^{\infty}(M)\to L^{2}(M)$ is also 
invariant with respect to $E$. Moreover, if $\sigma_S$ denotes the symbol of 
$S$ and $\sigma_T$ denotes the symbols of $T$ 
with respect to  the same orthonormal basis then 
\[\sigma_{S\circ T}=\sigma_S\sigma_T,\]
i.e. $\sigma_{S\circ T}(j)=\sigma_S(j)\sigma_T(j)$ for all $j\in\N_{0}.$
\end{prop}

Recalling Sobolev spaces $H^{s}(M)$ in \eqref{EQ:Sob-char} we have:

\begin{cor}\label{COR:Sobolev} 
Let an invariant linear operator $T:\cinfm\rightarrow C^{\infty}(M)$ 
have symbol $\sigma_{T}$ for which there exists $C>0$  and $m\in\mathbb R$ such that 
\[
\|\sigma_{T}(\ell)\|_{op}\leq C(1+\lambda_{\ell})^{\frac{m}{\nu}}
\]
holds for all $\ell\in\N_{0}$. Then $T$ extends to a bounded operator
from $H^{s}(M)$ to $H^{s-m}(M)$ for every $s\in\mathbb R$.
\end{cor} 
\begin{proof}
We note that by \eqref{EQ:Sob-char} the condition that 
$T:H^{s}(M)\to H^{s-m}(M)$ is bounded is equivalent to 
the condition that the operator
$$S:=(I+E)^{\frac{s-m}{\nu}}\circ T\circ (I+E)^{-\frac{s}{\nu}}$$ is bounded
on $L^{2}(M)$.
By Proposition \ref{comp1} and the fact that the powers of $E$ are
pseudo-differential operators with diagonal symbols, 
see \eqref{EQ:symbol-F}, we have
$$
\sigma_{S}(\ell)=(1+\lambda_{\ell})^{-\frac{m}{\nu}}\sigma_{T}(\ell).
$$
But then $\|\sigma_{S}(\ell)\|_{op}\leq C$ for all $\ell$ in view of the assumption on
$\sigma_{T}$, so that
the statement follows from Theorem \ref{L2}.
\end{proof}

\section{Schatten classes of operators on compact manifolds}
\label{SEC:Schatten-mfds}

In this section we give an application of the constructions in the previous
section to determine the membership of operators in Schatten classes
and then apply it to a particular family of operators on $L^{2}(M)$.

As a consequence of Theorem \ref{schchr-abstract}, 
we can now characterise invariant operators in Schatten classes on compact manifolds. 
We note that this characterisation does not assume any regularity of the kernel nor
of the symbol. Once we observe that the conditions for the membership
in the Schatten classes depend only on the basis $e_{j}^{k}$ and not on 
the operator $E$, we immediately obtain:

\begin{thm}\label{schchr} Let $0<r<\infty$. 
An invariant operator $T:L^2(M)\rightarrow L^2(M)$ is in $S_r(L^2(M))$ 
if and only if  $\sum\limits_{\ell=0}^{\infty}\|\sigma_{T}(\ell)\|_{S_r}^r<\infty$. 
Moreover
\[\|T\|_{S_r(L^2(M))}^r=\sum\limits_{\ell=0}^{\infty}\|\sigma_{T}(\ell)\|_{S_r}^r.\]
If an invariant operator $T:L^2(M)\rightarrow L^2(M)$ is in the trace class
$S_1(L^2(M))$, then 
\[\Tr(T)=\sum\limits_{\ell=0}^{\infty}\Tr(\sigma_{T}(\ell)).\]
\end{thm}

An interested reader can find in \cite{DRT:JMPA} further applications to Schatten classes and the so-called $r$-nuclearity of operators.


\end{document}